%% file: point2ellipse_series.tex
\documentclass[a4paper]{amsart}
\pdfoutput=1 

\usepackage{amssymb}

\usepackage{tikz}

\usepackage[backend = biber,
    style = authoryear,
    giveninits = true,  
    dashed = false,     
    uniquename = false,  
    maxcitenames = 2  
]{biblatex}
\DeclareCiteCommand{\cite}[\mkbibemph]
  {\usebibmacro{prenote}}%
  {\usebibmacro{citeindex}%
   \usebibmacro{cite}}
  {\multicitedelim}
  {\usebibmacro{postnote}}

\defbibfilter{appendixOnlyFilter}{
  segment=1 
  and not segment=0 
}
\AtBeginBibliography{\small}  

\usepackage{tabularx,ragged2e} 
\usepackage{hyperref}  
\usepackage{graphicx}  
\usepackage[font=footnotesize]{caption}   %
\usepackage{mathtools} 
\usepackage{longtable} 

\DeclareMathOperator{\atantwo}{atan2}

\title{Point-to-ellipse Fourier series}
\author{John-Olof Nilsson}
\address{John-Olof Nilsson is with Avioniq Sweden AB, e-mail: john\_nil@hotmail.com}
\date{}

\newtheorem{thm}{Theorem}

\begin{document}

\begin{abstract}
Fourier series with power series coefficients for the normal and distance to a point from an ellipse are derived. These expressions are the first of their kind and opens up a range of analysis and computational possibilities.
\end{abstract}

\maketitle

\vspace{-10mm}

\section{Introduction}

\begin{center}
\includegraphics[]{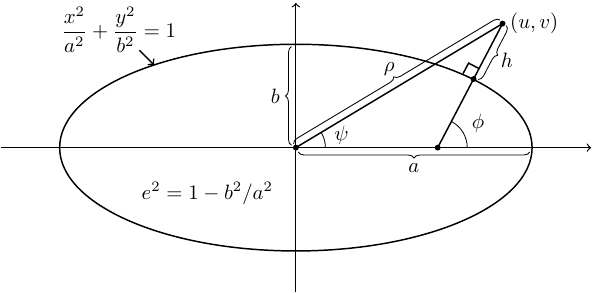}
\end{center}

\noindent
Determining the normal $\phi$ and the distance $h$ to a point $(u,v)$ from an ellipse with axes $a$ and $b$, as depicted above, is an extensively studied problem over the centuries:
Classical geometry techniques for the normal are known since the antiquity p.180~\cite{Heath1896};
Quartic equations
have been around since at least the early 1900-hundreds p.382~\cite{Gibson1911}, 
even if solid solution methods did not appear until recently~\cite{Vermeille2011};
100s of exact and approximate methods for computing $\phi$ and $h$ have been published since the 1960s~\cite{Nilsson2024}.
Despite this, only quartic equation, iterative, and closed-form approximation solutions are available, fundamentally limiting analysis and computations related to $\phi$ and $h$. 
%
In contrast, provided here are
Fourier and power series in $\psi$ and $\sin(\psi)$ for $\phi$, $h$, and $\sin(\phi)$ and $\cos(\phi)$ (vector normal) whose coefficients are, in turn, power series
and polynomials
in $a/\rho$ and $e^2$ with rational coefficients.
The series are valuable in that:
\begin{itemize}
\item They are fundamental results for the ellipse conic and provide the first general series expansions for the point-to-ellipse relation.
\item Point-to-ellipse being a fundamental relation means that they enable series expansions of many related quantities.
\item They are differentiable, making updates for e.g. ellipse fitting, coordinate transformations, and differential equation solutions conceivable.
\item Truncated series give separate simple algebraic approximations with potentially boundable errors and competitive performance.
\end{itemize}
Fourier series expansions have been attempted before by~\cite{Morrison1961} and \cite{Pick1967}, but these efforts have only resulted in a few initial and partial terms.
Series expansion inspiration is also drawn from~\cite{Nilsson2024minimax}.
The main limitation of this work is that no region of convergence is sought.
%
The series are given in the next section. A brief discussion about them follows.
Finally, derivations, tabulated coefficients, and implementations are found in Appendices~\ref{app:A} and~\ref{sec::tabulated_coefficients} and \url{https://github.com/jnil02/point2ellipse_series}, respectively.

\section{The series}
The series are provided below in Fourier multiple-angle and $\sin$-power series forms. The former is fundamental and useful from a theoretical perspective whilst the latter brings more structure by reducing the innermost series to a polynomial. Further, the latter form is also more useful from a computational perspective since truncations directly give polynomials in the ratios $\varrho = a/\rho$, $\sin(\psi) = v / \rho$, and $\cos(\psi) = u /\rho$.

\vspace{3mm}
\begin{thm}
For the point-to-ellipse relation and sufficiently small $ae^2/u$
\vspace{1mm}
\begin{align*}
\phi-\psi &=\sum_{n=1}^\infty\left(\sum_{k=1}^\infty\left(\sum_{l=\max(n,k)}^\infty c^{\phi}_{n,k,l}e^{2l}\right)\varrho^k\right)\sin(2n\psi)\\
\frac{\phi-\psi}{\cos(\psi)\sin(\psi)}&=\sum_{n=0}^{\infty} \left(\sum_{k=1}^\infty\left(\sum_{l=\max(n+1,k)}^{n+k} d^{\phi}_{n,k,l}  e^{2l}\right)\varrho^k\right)\sin^{2n}(\psi)\\[2mm]
\frac{h + a - \rho}{a}
&=\sum_{n=0}^\infty\left(\sum_{k=0}^\infty\left(\sum_{l=\max(n,k+1)}^\infty c^h_{n,k,l}e^{2l}\right)\varrho^k\right)\cos(2n\psi)\\
&=\sum_{n=1}^\infty\left(\sum_{k=0}^\infty\left(\sum_{l=\max(n,k+1)}^{n+k} d^h_{n,k,l}e^{2l}\right)\varrho^k\right)\sin^{2n}(\psi)\\[2mm]
\frac{\sin(\phi)}{\sin(\psi)}-1&=\sum_{n=0}^\infty\left( \sum_{k=1}^\infty\left(\sum_{l=\max(n,k)}^\infty c^{\sin}_{n,k,l}e^{2l}\right)\varrho^k\right)\cos(2n\psi)\\
&=\sum_{n=0}^\infty\left( \sum_{k=1}^\infty\left(\sum_{l=\max(n,k)}^{n+k} d^{\sin}_{n,k,l}e^{2l}\right)\varrho^k\right)\sin^{2n}(\psi)\\[2mm]
\frac{\cos(\phi)}{\cos(\psi)}-1&=\sum_{n=0}^\infty\left( \sum_{k=1}^\infty\left(\sum_{l=\max(n,k)}^\infty c^{\cos}_{n,k,l}e^{2l}\right)\varrho^k\right)\cos(2n\psi)\\
&=\sum_{n=1}^\infty\left( \sum_{k=1}^\infty\left(\sum_{l=\max(n,k)}^{n+k-1} d^{\cos}_{n,k,l}e^{2l}\right)\varrho^k\right)\sin^{2n}(\psi)
\end{align*}
\vspace{1mm}
where the coefficients $c^{*}_{n,k,l}$ and $d^{*}_{n,k,l}$ are rational and, in order, given by the expressions \eqref{eq:c_phi}, \eqref{eq:d_phi}, \eqref{eq:c_h}, \eqref{eq:d_h}, \eqref{eq:c_sin}, \eqref{eq:d_sin}, \eqref{eq:c_cos}, and \eqref{eq:d_cos}, respectively.
\end{thm}
\begin{proof}
See Appendix~\ref{app:A}.
\end{proof}
\vspace{3mm}

\noindent
Note, in practice, the region of convergence appear to be of useful size, even including $u=0$ but not too small $\varrho$. Potentially it is outside the ellipse evolute.
Further note, the $\cos(\psi)\sin(\psi)$ in the $\sin$-power series of $\phi-\psi$ could easily be integrated into the series, but this gives a poor rate of convergence, see~\ref{app:phi}. 
%
%
Finally, for computing both $\cos(\phi)$ and $\sin(\phi)$, it is preferable to compute them jointly with slightly different formulas, see~\ref{app:sincos}. For more computational details, see~\cite{Nilsson2024minimax}.

\section{Discussion}
To my knowledge, the series are the first general series expansions of $\phi$, $h$, $\sin(\phi)$, and $\cos(\phi)$ in terms of $\psi$ and $\rho$, or, alternatively, $u$ and $v$. Again, the series are:
\begin{itemize}
\item Essentially series solutions to the quartic latitude equation.
\item Series of fundamental relations, meaning that series for dependent quantities could be derived, enabling related analysis and computations. 
\item Trivially differentiable, both with respect to the point and the parameters, providing new paths for data fitting and transformation updates.
\item Easily approximated by truncation with granularly tuneable accuracy and potentially boundable errors for a given $\rho$. 
\end{itemize}
Hence, the series opens up a range of analysis and computational possibilities not obtainable from previous quartic equation, iterative, or closed-form approximate solutions.
Contributing to this are the $\sin$-power series variants, with its additional series structure. Depending on the use-case, $e^2$ may be a constant, e.g. geodesy, or a variable, e.g. ellipse fitting. 
For the latter case, this is obviously particularly valuable. For the former case, the $\sin$-power series structure is still valuable since $\sin^{2n}(\psi)$ is typically easier to evaluate for different $n$ than $\sin(2n\psi)$ and $\cos(2n\psi)$. 
%
%
Unfortunately, the series come with some disadvantages:
\begin{itemize}
\item They appear not to converge for all points and ellipses and the region of convergence is so far not clear.
\item The series appear to have geometric convergence with respect to $n$ but, due to the inner series, subgeometric convergence with respect to arithmetic operations, i.e. other methods have better asymptotic convergence. 
\end{itemize}
However, convergence can probably be clarified and complementary series in $\varrho^{-1}$ could possibly be derived.
Further, series accelerations may possibly be applied.
Finally, for many practical applications, $e^2$ is small or $\rho$ is large giving fast initial convergence and potentially competitive performance for required accuracies.



\section*{References}
\printbibliography[heading=none]

\appendix
\newrefsegment 

\section{Derivations}\label{app:A}
Let $(u,v)$ be a Cartesian coordinate of a point. Let $a$ and $b$ designate the major and minor axes of an ellipse aligned with the coordinate axes and centered at the origin. Let $h$ and $\phi$ designate the distance to the ellipse and the angle to the normal from the ellipse passing through $(u,v)$. Further, let $\rho=(u^2+v^2)^{1/2}$ and $\psi=\atantwo(v/u)$ be the polar coordinates of $(u,v)$. Then, from basic geometry,
\begin{align}
\label{eq:base1}
\rho\cos(\psi) &= (N + h)\cos(\phi)\\
\label{eq:base2}
\rho\sin(\psi) &= ((1-e^2)N + h)\sin(\phi)
\end{align}
where the eccentricity squared $e^2=1-b^2/a^2$ and the radius of curvature $N = a(1-e^2\sin(\phi))^{-1/2}$.
\eqref{eq:base1} and~\eqref{eq:base2} are the basis for the following series expansions derivations.
Further, in the derivation, the following formulas are used:
\newline
\emph{Lagrange reversion theorem}:
\begin{equation}
\begin{aligned}
\text{with}\quad y = x + d g(y),\quad\text{then for sufficiently small}\:d\\
f(y) = f(x) + \sum_{k=1}^{\infty}\frac{d^k}{k!}\left(\frac{\partial}{\partial x}\right)^{k-1}(f'(x)g^k(x))
\end{aligned}
\label{eq:lagrange}
\end{equation}
The \emph{general Leibniz rule}:
\begin{equation}
\frac{d^n}{dx^n}f(x)g(x) = \sum^n_{r=0}\binom{n}{r}\frac{d^{n-r}}{dx^{n-r}}f(x)\frac{d^r}{dx^r}g(x)\label{eq:leibniz}
\end{equation}
The \emph{Fa\'{a} di Bruno's formula}:
\begin{equation}
\frac{d^n}{dx^n}f(g(x))=\sum\frac{n!}{m_1!\dots m_n!} f^{(m_1+\dots+m_n)}(g(x))\prod_{i=1}^n\left(\frac{g^{(i)}(x)}{i!}\right)^{m_i}\label{eq:faa_di_bruno}
\end{equation}
%
Finally, for all but the $\phi$ expansion, the multi-angle Fourier expansions are obtained from the $\sin$-power expansions and the following relation. Let
\begin{equation*}
f(\phi) = \sum_{n=n_{\min}}^\infty \left(\sum_{k=k_{\min}}^\infty\left(\sum_{l=\max(n,k+k')}^{n+k+k''} e^{2l}d_{n,k,l}\right)\varrho^k\right)\sin^{2n}(\phi)
\end{equation*}
From the power reduction formulas
\begin{align*}
\sin^{2n}(\phi) &= \frac{1}{2^{2n}}\binom{2n}{n} + \frac{2}{2^{2n}}\sum_{i=0}^{n-1}(-1)^{n-i}\binom{2n}{i}\cos(2(n-i)\phi)
\end{align*}
Rearranging the sums gives
\begin{equation}
f(\phi)=\sum^\infty_{n=0}\left(\sum_{k=k_{\min}}^\infty\left(\sum_{l=\max(\max(n,n_{\min}),k+k')}^\infty e^{2l}c_{n,k,l}\right)\varrho^k\right)\cos(2n\phi)\label{eq:sinPowToCosMul}
\end{equation}
where 
\begin{equation*}
c_{n,k,l}=\sum_{i=\max(\max(n,n_{\min}),l-k-k'')}^l d_{i,k,l}\frac{2}{2^{2i+\delta_n}}(-1)^{n}\binom{2i}{i-n}
\end{equation*}
where $\delta_*$ is the \emph{Kronecker delta}.
In the following subsections, expansions are derived for the desirable quantities and and some auxiliary quantities.

\subsection{Expansions for $\phi$}\label{app:phi}
Let $\varrho=a/\rho$. Dividing \eqref{eq:base2} by \eqref{eq:base1} and solving for $\tan(\phi)$
\begin{equation}\label{ref:base3}
\tan(\phi) = \tan(\psi) + \varrho\frac{e^2}{\cos(\psi)}\frac{\sin(\phi)}{(1 - e^2\sin^2(\phi))^{1/2}}
\end{equation}
Using the \emph{Lagrange reversion theorem}~\eqref{eq:lagrange} as in~\cite{Morrison1961} with
\begin{equation*}
y = \tan(\phi),\quad x=\tan(\psi)
\end{equation*}
\begin{equation*}
g(y)=\frac{y}{(1+(1-e^2)y^2)^{1/2}}=\frac{\sin(\phi)}{(1 - e^2\sin^2(\phi))^{1/2}}
\end{equation*}
\begin{equation*}
d = \varrho e^2(1+x^2)^{1/2}=\frac{\varrho e^2}{\cos(\psi)}
\end{equation*}
\begin{equation*}
\text{and}\quad f(y)=\tan^{-1}(y)
\end{equation*}
gives
\begin{equation*}
\tan^{-1}(y) = \tan^{-1}(x) + \sum_{k=1}^{\infty}\varrho^k\frac{e^{2k}(1+x^2)^{k/2}}{k!}\left(\frac{\partial}{\partial x}\right)^{k-1}\frac{x^k}{1+x^2}\frac{1}{(1+(1-e^2)x^2)^{k/2}}
\end{equation*}
To proceed beyond~\cite{Morrison1961}, an expression for the n\textsuperscript{th} derivative of the latter part is needed. Further, series expansions of $h$ and extra structure is enabled by derivation of alternative $\sin$-power series expansions. Finally, this expansion holds for \emph{sufficiently small} $d = \varrho e^2/\cos(\psi) = ae^2/u$.

The n\textsuperscript{th} derivative of $x^k$
\begin{equation*}
\frac{d^n}{dx^n}x^k =
\begin{cases}
\frac{k!}{(k-n)!}x^{k-n} & :n\leq k \\
0 & :n > k
\end{cases}
\end{equation*}
The n\textsuperscript{th} derivative of $\frac{1}{1+x^2}$
\begin{align*}
\frac{d^n}{dx^n}&\frac{1}{1+x^2} = \frac{d^n}{dx^n}\frac{1}{2i}\left(\frac{1}{x-i} - \frac{1}{x+i}\right)\\
&=(-1)^n\frac{n!}{2i}\left(\frac{1}{(x-i)^{n+1}} - \frac{1}{(x+i)^{n+1}}\right)\\
&=(-1)^n\frac{n!}{2i(x^2 + 1)^{n+1}}\left((x+i)^{n+1}-(x-i)^{n+1}\right)\\
&=(-1)^n\frac{n!}{2i(x^2 + 1)^{n+1}}\left(\sum^{n+1}_{k=0}\binom{n+1}{k}x^{n+1-k}i^k - \sum^{n+1}_{k=0}\binom{n+1}{k}x^{n+1-k}(-i)^k\right)\\
&=(-1)^n\frac{n!}{2i(x^2 + 1)^{n+1}}\sum^{n+1}_{k=0}\binom{n+1}{k}x^{n+1-k}i^k\left(1 - (-1)^k\right)\\
&=\frac{n!}{(x^2 + 1)^{n+1}}\sum^{n+1}_{k\in\mathbf{N}_o}\binom{n+1}{k}x^{n+1-k}(-1)^{n+(k-1)/2}\\
&=\sum^{n+1}_{k\in\mathbf{N}_o}(-1)^{n+(k-1)/2}n!\binom{n+1}{k}\frac{x^{n+1-k}}{(x^2 + 1)^{n+1}}\\
&=\sum^{\lfloor n/2\rfloor}_{k=0}(-1)^{n+k}n!\binom{n+1}{2k+1}\frac{x^{n-2k}}{(x^2 + 1)^{n+1}}
\end{align*}
where $\mathbf{N}_o$ designates the odd natural numbers.
Combining the expressions for $x^k$ and $\tfrac{1}{\smash{1+x^{{}_2}}}$ with the \emph{general Leibniz rule}~\eqref{eq:leibniz} gives the n\textsuperscript{th} derivative of $\tfrac{\smash{x^k}}{\smash{1+x^{{}_2}}}$
\begin{align*}
\frac{d^n}{dx^n}\frac{x^k}{1+x^2} &= \sum^n_{m=0}\binom{n}{m}\frac{k!}{(k-(n-m))!}x^{k-(n-m)}\\
&\quad\quad\sum^{\lfloor m/2\rfloor}_{p=0}(-1)^{m+p}m!\binom{m+1}{2p+1}\frac{x^{m-2p}}{(1 + x^2)^{m+1}}\\
& = \sum^n_{m=0}\sum^{\lfloor m/2\rfloor}_{p=0}\binom{m+1}{2p+1}\binom{k}{n-m}n!(-1)^{m+p}\frac{x^{m+1}}{(1 + x^2)^{m+1}}x^{k-n+m-2p-1}
\end{align*}
Required next is the n\textsuperscript{th} derivative of $\frac{1}{(1+(1-e^2)x^2)^{k/2}}$.
Trivially
\begin{equation*}
\frac{d^n}{dx^n}\frac{1}{(1+x)^{k/2}} = (-1)^n(k/2)^{\overline{n}}\frac{1}{(1+x)^{k/2+n}}
\end{equation*}
where $(\cdot)^{\overline{n}}$ indicates the rising factorial. Further
\begin{equation*}
\frac{d^n}{dx^n}(1-e^2)x^2 =
\begin{cases}
0 & \text{if}\quad n > 2 \\
\frac{2!}{(2-n)!}(1-e^2)x^{2-n} &  \text{otherwise}
\end{cases}
\end{equation*}
Plugging this into the \emph{Fa\'{a} di Bruno's formula}~\eqref{eq:faa_di_bruno} where the sum is over all partitions $m_1+2m_2+\dots+nm_n = n$, with $f(x)=\frac{1}{\smash{(1+x)^{{}_{k/2}}}}$ and $g(x)=(1-e^2)x^2$ and noting that the derivative of $g(x)=(1-e^2)x^2$ equals zero for $n>2$ which means that $m_n=0\forall n>2$, $m_1+2m_2=n$ and $m_2\in[0,\lfloor n/2\rfloor]$, giving
\begin{align*}
\frac{d^n}{dx^n}&\frac{1}{(1+(1-e^2)x^2)^{k/2}}\\
&=\sum\frac{n!}{m_1!m_2!} (-1)^{m_1+m_2}(k/2)^{\overline{m_1+m_2}}\frac{(2(1-e^2)x)^{m_1}(1-e^2)^{m_2}}{(1+(1-e^2)x^2)^{k/2+(m_1+m_2)}}\\
&=\sum_{m_2=0}^{\lfloor n/2\rfloor}\frac{n!(k/2)^{\overline{n-m_2}}}{(n-2m_2)!m_2!} (-1)^{n-m_2}\frac{(2(1-e^2)x)^{(n-2m_2)}(1-e^2)^{m_2}}{(1+(1-e^2)x^2)^{k/2+(n-m_2)}}\\
&=\sum_{m=0}^{\lfloor n/2\rfloor}(-1)^{n-m}\frac{n!(k/2)^{\overline{n-m}}2^{(n-2m)}}{(n-2m)!m!}\frac{(1-e^2)^{(n-m)}x^{(n-2m)}}{(1+(1-e^2)x^2)^{k/2+(n-m)}}
\end{align*}
Again, applying the \emph{general Leibniz rule} with the factors $\frac{\smash{x^k}}{\smash{1+x^{{}_2}}}$ and $\frac{1}{\smash{(1 + (1-e^{{}_2})x^{{}_2})^{{}_{k/2}}}}$ and expanding $(1-e^2)^{r-q}$ in a binomial sum gives the k\textsuperscript{th} term of the Lagrange reversion expansion
\begin{align}
&\varrho^k\frac{e^{2k}(1+x^2)^{k/2}}{k!}\left(\frac{\partial}{\partial x}\right)^{k-1}\frac{x^k}{1+x^2}\frac{1}{(1+(1-e^2)x^2)^{k/2}}\nonumber\\
&\quad\quad= \varrho^k\frac{e^{2k}(1+x^2)^{k/2}}{k!}\sum^{k-1}_{r=0}\binom{k-1}{r}\frac{d^{k-1-r}}{dx^{k-1-r}}\frac{x^k}{1+x^2}\frac{d^r}{dx^r}\frac{1}{(1 + (1-e^2)x^2)^{k/2}}\nonumber\\
&\quad\quad=\varrho^k\frac{e^{2k}(1+x^2)^{k/2}}{k!}\sum^{k-1}_{r=0}\binom{k-1}{r}\sum^{k-1-r}_{m=0}\sum^{\lfloor m/2\rfloor}_{p=0}\binom{m+1}{2p+1}\binom{k}{k-1-r-m}\nonumber\\
&\quad\quad\quad\quad\quad\quad(k-1-r)!(-1)^{m+p}\frac{x^{m+1}}{(1 + x^2)^{m+1}}x^{k-(k-1-r)+m-2p-1}\nonumber\\
&\quad\quad\quad\quad\quad\quad\quad\quad\quad\quad\sum_{q=0}^{\lfloor r/2\rfloor}(-1)^{r-q}\frac{r!(k/2)^{\overline{r-q}}2^{(r-2q)}}{(r-2q)!q!}\frac{(1-e^2)^{(r-q)}x^{(r-2q)}}{(1+(1-e^2)x^2)^{k/2+(r-q)}}\nonumber\\
&\quad\quad=\varrho^k\sum^{k-1}_{r=0}\sum^{k-1-r}_{m=0}\sum^{\lfloor m/2\rfloor}_{p=0}\sum_{q=0}^{\lfloor r/2\rfloor}\sum_{t=0}^{r-q}(-1)^{r-q+m+p+t}\frac{(k/2)^{\overline{r-q}}2^{(r-2q)}e^{2(t+k)}}{(r-2q)!q!(1+r+m)}\binom{r-q}{t}\raisetag{-8.75mm}\label{eq:lagrange_tot}\\
&\quad\quad\quad\binom{k-1}{r+m}\binom{m+1}{2p+1}x^{-2p-1}\!\left(\frac{x^2}{1 + x^2}\right)^{m+1-k/2}\!\left(\frac{x}{(1+(1-e^2)x^2)^{1/2}}\right)^{k+2(r-q)}\nonumber
\end{align}
Substitute $\tan(\phi)$ back for $x$, applying binomial expansion of $\frac{1}{(1-e^2\sin^2(\phi))}$, and collecting $\sin$ and $\cos$ factors gives the $x$ factors of~\eqref{eq:lagrange_tot}
\begin{align}
\nonumber
x^{-(2p+1)}&\left(\frac{x^2}{1 + x^2}\right)^{m+1-k/2}\left(\frac{x}{(1+(1-e^2)x^2)^{1/2}}\right)^{k+2(r-q)}\\
\nonumber
&=\tan^{-(2p+1)}(\phi)\sin^{2(m+1-k/2)}(\phi)\frac{\sin^{k+2(r-q)}(\phi)}{(1-e^2\sin^2(\phi))^{(k+2(r-q))/2}}\\
\label{eq:lat_split}
&=\sum_{s=0}^\infty\binom{\frac{k}{2}+r-q+s-1}{s}e^{2s}\cos^{2p+1}(\phi)\sin^{2(m+r-q+s-p)+1}(\phi)
\end{align}
From here the derivations of the Fourier multi-angle and the $\sin$-power series forms split. For the Fourier multi-angle, proceed from~\eqref{eq:lat_split} by applying the trigonometric power-reduction and product-to-sum formulas
\begin{align*}
x^{-(2p+1)}&\left(\frac{x^2}{1 + x^2}\right)^{m+1-k/2}\left(\frac{x}{(1+(1-e^2)x^2)^{1/2}}\right)^{k+2(r-q)}\\
&=\sum_{s=0}^\infty\binom{\frac{k}{2}+r-q+s-1}{s}e^{2s}
\frac{2}{2^{2p+1}}\sum_i^{p}\binom{2p+1}{i}\cos((2p+1-2i)\phi)\\
&\quad\quad\frac{2}{2^{2(m+r-q+s-p)+1}}\sum_j^{m+r-q+s-p}(-1)^{m+r-q+s-p-j}\\
&\quad\quad\binom{2(m+r-q+s-p)+1}{j}\sin((2(m+r-q+s-p)+1-2j)\phi)\\
&=\sum_{s=0}^\infty e^{2s}\sum_i^{p}\sum_j^w\binom{\frac{k}{2}+r-q+s-1}{s}\binom{2p+1}{i}\binom{2w+1}{j}
\frac{(-1)^{w-j}}{2^{2(w+p)+1}}\\
&\quad\quad\big(\sin(2(w-j+p-i+1)\phi) + \sin(2(w-j-(p-i))\phi)\big)
\end{align*}
where $w=m+r-q+s-p$. Inserting back into \eqref{eq:lagrange_tot} gives $\phi-\psi$ in terms of $\sin$-multiples
\begin{align*}
\phi-\psi = &\sum_{k=1}^\infty\sum_{s=0}^\infty\sum^{k-1}_{r=0}\sum^{k-1-r}_{m=0}\sum^{\lfloor m/2\rfloor}_{p=0}\sum_{q=0}^{\lfloor r/2\rfloor}\!\sum_{t=0}^{r-q}\sum_{i=0}^{p}\sum_{j=0}^{w}\varrho^k\frac{(-1)^{s-j+t}(k/2)^{\overline{r-q}}}{q!(r\!-\!2q)!(m\!+\!1\!+\!r)}\frac{e^{2(k+s+t)}}{2^{2(m+s)+r+1}}\\
&\binom{r-q}{t}\binom{k-1}{m+r}\binom{m+1}{2p+1}\binom{\frac{k}{2}+r-q+s-1}{s}\binom{2p+1}{i}\binom{2w+1}{j}\\
&\big(\sin(2(w-j+p-i+1)\phi) + \sin(2(w-j-(p-i))\phi)\big)
\end{align*}
To collect equal terms of $\sin(2n\psi)$, $e^{2l}$ and $\varrho^{k}$, note
\begin{itemize}
\item $w-j+p-i+1 \neq w-j-(p-i)$ so there is no $n=0$ term.
\item $l=s+k+t$ meaning that $s=l-k-t$.
\item $t$ is limited by $l$ since obviously $t+k\leq l$.
\item $w-j+p-i+1\in [1, k+s]$ meaning the first term can only contributes with a positive angle.
\item $w-j+p-i+1=n$ imply $j=w+p-i+1-n$ which, together with $j\in[0,w]$ and $i\in[0,p]$, imply $i\in[\max(0,p-n+1), \min(p, w+p-n+1)]$.
\item $w-j-(p-i)\in [-\lceil (k-1)/2\rceil, k+s-1]$ so the second term can contribute with a positive and a negative term.
\item $w-j-(p-i)=n$ imply $j = w-p+i-n$ which, together with $j\in[0,w]$ and $i\in[0,p]$, imply $i\in[p-w+n, p]$.
\item $w-j-(p-i)=-n$ imply $j = w-p+i+n$ which, together with $j\in[0,w]$ and $i\in[0,p]$, imply $i\in[p-w-n, p-n]$.
\item $l$ is obviously below limited by $k$, but also by $n$, since for a given $n$, $l\geq k+s\geq n$.
\end{itemize}
%
giving
\begin{align*}
\phi-\psi &= \sum_{n=1}^\infty\left(\sum_{k=1}^\infty\left(\sum_{l=\max(n,k)}^\infty c^\phi_{n,k,l}e^{2l}\right)\varrho^k\right)\sin(2n\phi)
\end{align*}
where the coefficients are given by
\begin{align}
c^\phi_{n,k,l}&=\sum^{k-1}_{r=0}\sum^{k-1-r}_{m=0}\sum^{\lfloor m/2\rfloor}_{p=0}\sum_{q=0}^{\lfloor r/2\rfloor}\sum_{t=0}^{\min(l-k,r-q)}\frac{(-1)^{l-k}(k/2)^{\overline{r-q}}}{q!(r-2q)!(m+1+r)2^{2(m+l-k-t)+r+1}}\nonumber\\
&\quad\quad\quad\quad\binom{r-q}{t}\binom{k-1}{m+r}\binom{m+1}{2p+1}\binom{\frac{k}{2}+r-q+l-k-t-1}{l-k-t}\nonumber\\
&\left(\sum_{i=\max(0,p-n+1)}^{\min(p, p-n+1+w)}(-1)^{w+p-i+1-n}\binom{2p+1}{i}\binom{2w+1}{w+p-i+1-n}\right.\nonumber\\
&\quad\quad\quad\quad\quad+\sum_{i=p-w+n}^{p}(-1)^{w-p+i-n}\binom{2p+1}{i}\binom{2w+1}{w-p+i-n}\nonumber\\
&\left.\quad\quad\quad\quad\quad\quad\quad\quad-\quad \sum_{i=p-w-n}^{p-n}(-1)^{w-p+i+n}\binom{2p+1}{i}\binom{2w+1}{w-p+i+n}\right)\label{eq:c_phi}
\end{align}
%
In turn, the $\sin$-power series form is found by proceeding from~\eqref{eq:lat_split} as follows.
\begin{align*}
&x^{-(2p+1)}\left(\frac{x^2}{1 + x^2}\right)^{m+1-k/2}\left(\frac{x}{(1+(1-e^2)x^2)^{1/2}}\right)^{k+2(r-q)}\\
&=\cos(\phi)\sum_{s=0}^\infty\binom{\frac{k}{2}+r-q+s-1}{s}e^{2s}(1-\sin^{2}(\phi))^p\sin^{2(m+r-q+s-p)+1}(\phi)\\
&=\cos(\phi)\sum_{s=0}^\infty\binom{\frac{k}{2}+r-q+s-1}{s}e^{2s}\sum_{i=0}^{p}\binom{p}{i}(-1)^i\sin^{2i}(\phi)\sin^{2(m+r-q+s-p)+1}(\phi)\\
&=\cos(\phi)\sum_{s=0}^\infty e^{2s}\sum_{i=0}^{p}(-1)^i\binom{\frac{k}{2}+r-q+s-1}{s}\binom{p}{i}\sin^{2(m+r-q+s-p+i)+1}(\phi)\\
\end{align*}
Inserting back into~\eqref{eq:lagrange_tot} gives $\phi-\psi$ in terms of $\sin$-powers
\begin{align*}
\phi-\psi &= \cos(\phi)\sum_{k=1}^\infty\sum_{s=0}^\infty\sum^{k-1}_{r=0}\sum^{k-1-r}_{m=0}\sum^{\lfloor m/2\rfloor}_{p=0}\sum_{q=0}^{\lfloor r/2\rfloor}\sum_{t=0}^{r-q}\sum_{i=0}^{p}\varrho^k e^{2(k+s+t)}\\
&\quad\binom{r-q}{t}\binom{k-1}{m+r}\binom{m+1}{2p+1}\binom{\frac{k}{2}+r-q+s-1}{s}\binom{p}{i}\\
&\quad\frac{(-1)^{m+r-q+p+t+i}(k/2)^{\overline{r-q}}2^{r-2q}}{q!(r-2q)!(m+1+r)}\sin^{2(m+r-q+s-p+i)+1}(\phi)\\
\end{align*}
To collect equal terms of $\sin^{2n+1}(\psi)$, $e^{2l}$ and $\varrho^{k}$, note
\begin{itemize}
\item $l=s+k+t$ meaning that $s=l-k-t$.
\item $t$ is limited by $l$ and $k$ since obviously $t+k\leq l$.
\item $l$ is obviously below limited by $k$.
\item For a given $n$, from the $\sin$ factor $i = n - (m+r-q+s-p)$
\item For a given $n$ and $k$, $l$ is above limited since $n - (m+r-q+s-p) = i \geq 0$ implying $l \leq n+k$.
\item For a given $k$, $l$ is below limited since $p \geq i = n - (m+r-q+s-p)$ implying $l \geq n+1$.
\end{itemize}
giving
\begin{equation}\label{eq:latSinPow}
\frac{\phi-\psi}{\cos(\psi)\sin(\psi)} =\sum_{n=0}^{\infty} \left(\sum_{k=1}^\infty\left(\sum_{l=\max(n+1,k)}^{n+k} d^\phi_{n,k,l} e^{2l}\right)\varrho^k\right)\sin^{2n}(\psi)
\end{equation}
where
\begin{align}
d^\phi_{n,k,l}&=\sum^{k-1}_{r=0}\sum^{k-1-r}_{m=0}\sum^{\lfloor m/2\rfloor}_{p=0}\sum_{q=0}^{\lfloor r/2\rfloor}\sum_{t=0}^{\min(l-k,r-q)}\frac{(-1)^{2(p+t)+n-l+k}(k/2)^{\overline{r-q}}2^{r-2q}}{q!(r-2q)!(m+1+r)}\nonumber\\
&\quad\binom{r\!-\!q}{t}\!\binom{k\!-\!1}{m\!+\!r}\!\binom{m\!+\!1}{2p\!+\!1}\!\binom{\frac{k}{2}\!+\!r\!-\!q\!+\!l\!-\!k\!-\!t\!-\!1}{l-k-t}\!\binom{p}{n\!-\!m\!-\!r\!+\!q\!-\!l\!+\!k\!+\!t\!+\!p}\raisetag{-7mm}\label{eq:d_phi}
\end{align}
Note, $\cos(\psi)$ in \eqref{eq:latSinPow} (an obiously $\sin(\psi)$) can easily be integrated into the series, by replacing the lower limit of $l$ with $k$ rather than $\max(n+1,k)$ and replacing $p$ with $p+\frac{1}{2}$ in the last binomial of \eqref{eq:d_phi}. However, this gives a series whose convergence is limited by the expansion of $\cos(\psi)$. It is typically better to use $\cos(\psi)=(1-\smash{\sin}^2(\psi))^{1/2}$ and, therefore, such a series is not further explored.

\subsection{Expansions for $(\phi-\psi)^i$}
For brevity, let
\begin{gather*}
a_{n,k}=\sum_{l=\max(n+1,k)}^{n+k} d^\phi_{n,k,l} e^{2l}\quad\text{and}\quad
a_n =\sum_{k=1}^\infty a_{n,k}\varrho^k
\end{gather*}
Then, from~\eqref{eq:latSinPow}
\begin{align*}
(\phi-\psi)^i&=\cos^i(\psi)\sin^i(\psi)\left(\sum_{n=0}^\infty a_n\sin^{2n}(\psi)\right)^i\\
 &=\cos^i(\psi)\sin^i\sum_{n=0}^\infty b_{n,i}\sin^{2n}(\psi)
\end{align*}
where $b_{n,i}$ is a polynomial in $a_0,\dots,a_n$. The recurrence formula of powers of power series with non-zero constant term~\cite{Gould1974}, implicitly defines the \emph{ordinary potential polynomials}
\begin{align*}
b_{n,i} &=A_{n,i}(a_0,\dots,a_n) \\
&=
\begin{cases}
{a_0}^i& n=0\\
 \frac{1}{na_0}\sum_{k=1}^n(ki-n+k){a_k}b_{n-k,i}& n> 0
\end{cases}
\end{align*}
Further, since the series of $a_l$ has a zero constant term, its powers~\cite{Taghavian2023}
\begin{align*}
a_l^{i_l}=
\begin{cases}
1 & i_l = 0 \\
\sum_{k=i_l}^\infty \hat{B}_{k,i_l}(a_{l,1},\dots,a_{l,k-i_l+1})\varrho^k & i_l>0
\end{cases}
\end{align*}
where $\hat{B}_{k,i_l}(a_{l,1},\dots,a_{l,k-i_l+1})$ are the \emph{partial ordinary Bell polynomials} with a similar recurrence relation
\begin{equation*}
\hat{B}_{k,i}(x_1,\dots,x_{k-i+1})=
\begin{cases}
\delta_k & i = 0 \\
\sum_{j=1}^{k-i+1}x_j\hat{B}_{k-j,i-1}(x_1,\dots,x_{k-j-i}) & k \geq i, i \geq 1
\end{cases}
\end{equation*}
Finally, to collect equal terms of $e^{2l}$ and $\varrho^{k}$ in $b_{n,i}$, note
\begin{itemize}
\item For a given $i$, since $k$ starts at 1, the lowest power of $\varrho$ is $\varrho^i$.
\item For a given $k$, for each $a_{n,k}$ $l\geq k$ meaning that the power $e^{2l}$ is lower bounded by $k$.
\item For a given $n$, the lowest power of $e^2$ comes from multiplying the lowest powers in $a_l$. Writing $b_{n,i}$ as
\begin{equation*}
A_{n,i}(a_0,\dots,a_n)=b_{n,i}=\sum_{\substack{i_0+\dots+i_n=i \\ i_1+2i_2+\dots+ni_n=n}}\binom{i}{i_0,\dots,i_n}\prod_{l=0}^{i}a_l^{i_l}
\end{equation*}
and replacing each $a_l$ with the lowest power $e^{2(n+1)}$ gives
\begin{equation*}
\prod_{l=0}^{i}a_l^{i_l}\sim e^{2(i_0 + 2i_1+3i_2 + \dots+(n+1)i_n)}=e^{2(n+i)}
\end{equation*}
meaning that the power of $e^{2l}$ is lower bounded by $n+i$. 
\item For a given $n$ and $k$, the highest power of $e^2$ is trivially bounded by $n+k$.
\end{itemize}
Consequently
\begin{equation*}
b_{n,i} = \sum_{k=i}^\infty\left(\sum_{l=\max(n+i,k)}^{n+k} c^\phi_{n,k,l,i}e^{2l}\right)\varrho^k
\end{equation*}
where
\begin{align*}
d^\phi_{n,k,l,i}&=[e^{2l}][\varrho^k]\left(A_ {n,i}(a_0,\dots,a_n)\left[a_n^{i_n}\rightarrow\sum_{k=i_n}^\infty \hat{B}_{k,i_n}(a_{n,1},\dots,a_{n,k-i_n+1})\varrho^k\right]\right)\\
&=[e^{2l}][\varrho^k]\left(\rule{0cm}{7mm}A_ {n,i}(a_0,\dots,a_n)\left[\rule{0cm}{7mm}a_n^{i_n}\rightarrow\right.\right.\\
&\!\!\!\!\left.\left.\sum_{k=i_n}^\infty \hat{B}_{k,i_n}\left(\sum_{l=\max(1,n+1)}^{n+1} d^\phi_{n,1,l} e^{2l},\dots,\sum_{l=\max(n+1,k-i_n+1)}^{n+k-i_n+1} d^\phi_{n,k-i_n+1,l} e^{2l}\right)\varrho^k\right]\right)
\end{align*}
where $[x^k]f(x)$ means the coefficient of $x^k$ in the series or polynomial $f(x)$ and $f(x)[x\rightarrow y]$ means the series or polynomial $f(x)$ with $x$ substituted with $y$. Note, the substitution of $a_n^{i_n}$ is done for all $n$ and $i_n$. Further, note, in determining $[x^k]f(x)$ in the coefficient expression above, the polynomials of the series have to be expanded with Cauchy products. 
Combining it all gives
\begin{align}\label{eq:phiDiffPow}
(\phi-\psi)^i
 &=\cos^i(\psi)\sin^i\sum_{n=0}^\infty \left(\sum_{k=i}^\infty\left(\sum_{l=\max(n+i,k)}^{n+k} d^\phi_{n,k,l,i}e^{2l}\right)\varrho^k\right)\sin^{2n}(\psi)
\end{align}

\subsection{Expansions for $\sin(\phi)$ and $\cos(\phi)$}\label{app:sincos}
Expanding $\sin(\phi)$ in a Taylor series around $\psi$, using the $\sin$-power series expansion for $(\phi-\psi)^i$ from \eqref{eq:phiDiffPow} and expanding $\cos^{2x}(\psi)$ in $\sin^2(\psi)$ gives
\begin{align}
\nonumber\frac{\sin(\phi)}{\sin(\psi)}&=\sum_{i=0}^\infty\frac{\partial}{\partial\varphi^i}\left.\frac{\sin(\varphi)}{i!}\right|_{\psi}\frac{1}{\sin(\psi)}(\phi-\psi)^i\\
\label{eq:sin_split}\raisetag{-7.5mm}&=\sum_{i\in\mathbf{N}_e}^\infty\frac{(-1)^{i/2}}{i!}(\phi-\psi)^i+\sum_{i\in\mathbf{N}_o}^\infty\frac{(-1)^{(i-1)/2}}{i!}\frac{\cos(\psi)}{\sin(\psi)}(\phi-\psi)^i\\
\nonumber&=1+\sum_{i=1}^\infty\frac{(-1)^{\lfloor i/2\rfloor}}{i!}\cos^{2\lceil i/2\rceil}(\psi)\sin^{2\lfloor i/2\rfloor}\sum_{n=0}^\infty b_{n,i}\sin^{2n}(\psi)\\
\nonumber&=1+\sum_{i=1}^\infty\sum_{j=0}^{\lceil i/2\rceil}\binom{\lceil i/2\rceil}{j}\frac{(-1)^{\lfloor i/2\rfloor+j}}{i!}\sum_{n=0}^\infty b_{n,i}\sin^{2(n+\lfloor i/2\rfloor+j)}(\psi)\\
\nonumber&\!\!\!\!\!\!\!\!\!\!\!\!\!\!\!\!\!\!=1\!+\!\sum_{i=1}^\infty\!\sum_{j=0}^{\lceil i/2\rceil}\!\!\binom{\lceil i/2\rceil}{j}\frac{(-1)^{\lfloor i/2\rfloor\!+\!j}}{i!}\!\sum_{n=0}^\infty\! \left(\!\sum_{k=i}^\infty\!\left(\sum_{l=\max(n+i,k)}^{n+k}\!\!\!\!\!\!\!\!\! e^{2l}d^\phi_{n,k,l,i}\!\right)\!\varrho^k\!\right)\!\sin^{2(n\!+\!\lfloor i/2\rfloor\!+\!j)}(\psi)\\
\label{eq:sin}\raisetag{-9mm}&=1+\sum_{n=0}^\infty \left(\sum_{k=1}^\infty\left(\sum_{l=\max(n,k)}^{n+k} d^{\sin}_{n,k,l}e^{2l}\right)\varrho^k\right)\sin^{2n}(\psi)
\end{align}
where $\mathbf{N}_e$ and $\mathbf{N}_o$ designate the even and odd natural numbers and where
\begin{align}
\nonumber &d^{\sin}_{n',k',l'}\\
\nonumber &=\sum_{i=1}^\infty\!\sum_{j=0}^{\lceil i/2\rceil}\!\binom{\lceil i/2\rceil}{j}\frac{(-1)^{\lfloor i/2\rfloor\!+\!j}}{i!}\!\sum_{n=0}^\infty\! \left(\sum_{k=i}^\infty\!\left(\sum_{l=\max(n+i,k)}^{n+k}\!\!\!\!\! \delta_{l-l'}d^\phi_{n,k,l,i}\!\right)\!\delta_{k-k'}\!\right)\!\delta_{n\!+\!\lfloor i/2\rfloor\!+\!j-n'}\\
\raisetag{-9mm}&\quad\quad\quad=\sum_{i=1}^{\min(k',2n'+1)}\sum_{j=\max(0,\lceil i/2\rceil-l'+n')}^{\min(\lceil i/2\rceil,n'-\lfloor i/2\rfloor)}\binom{\lceil i/2\rceil}{j}\frac{(-1)^{\lfloor i/2\rfloor+j}}{i!} d^\phi_{n'-\lfloor i/2\rfloor-j,k',l',i}\label{eq:d_sin}
\end{align}
where 
the summation limits are due to
\begin{itemize}
\item $\delta_{n+\lfloor i/2\rfloor+j-n'}$, $i\geq 1$, $j\geq 0$ and $n\geq 0$ imply
\begin{itemize}
\item $n'\geq 0$
\item $i\leq 2n'+1$
\item $j\leq n'-\lfloor i/2\rfloor$
\item $n = n'-\lfloor i/2\rfloor-j$
\end{itemize}
\item $\delta_{k-k'}$, $i\geq 1$ and $k\geq i$ imply
\begin{itemize}
\item $k'=k$
\item $k'\geq 1$
\item $i\leq k'$
\end{itemize}
\item $\delta_{l-l'}$ and $\max(n+i,k)\leq l\leq n+k$, imply
\begin{itemize}
\item $l'=l$
\item $i\leq l'-n= l'-n'+\lfloor i/2\rfloor+j$ meaning $\lceil i/2\rceil-l'+n'\leq j$
\item $\max(n',k')\leq l'\leq n'+k'$
\end{itemize}
\end{itemize}
Using~\eqref{eq:sinPowToCosMul}, the Fourier series follows as
\begin{equation}
\frac{\sin(\phi)}{\sin(\psi)}-1=\sum_{n=0}^\infty\left( \sum_{k=1}^\infty\left(\sum_{l=\max(n,k)}^\infty c^{\sin}_{n,k,l}e^{2l}\right)\varrho^k\right)\cos(2n\psi)
\end{equation}
where
\begin{equation}
c^{\sin}_{n,k,l}=\sum_{i=\max(n,l-k)}^l d^{\sin}_{i,k,l}\frac{2}{2^{2i+\delta_n}}(-1)^{n}\binom{2i}{i-n}\label{eq:c_sin}
\end{equation}

The derivation for $\cos(\phi)$ is identical with $\lfloor\cdot\rfloor$ changed to $\lceil\cdot\rceil$ and \emph{vice versa}, i.e.
\begin{align}
\nonumber\frac{\cos(\phi)}{\cos(\psi)}&=\sum_{i=0}^\infty\frac{\partial}{\partial\varphi^i}\left.\frac{\cos(\varphi)}{i!}\right|_{\psi}\frac{1}{\cos(\psi)}(\phi-\psi)^i\\
\label{eq:cos_split}&=\sum_{i\in\mathbf{N}_e}^\infty\frac{(-1)^{i/2}}{i!}(\phi-\psi)^i-\frac{\sin^2(\psi)}{\cos^2(\psi)}\sum_{i\in\mathbf{N}_o}^\infty\frac{(-1)^{(i-1)/2}}{i!}\frac{\cos(\psi)}{\sin(\psi)}(\phi-\psi)^i\\
\label{eq:cos}&=1+\sum_{n=0}^\infty \left(\sum_{k=1}^\infty\left(\sum_{l=\max(n,k)}^{n+k-1} d^{\cos}_{n,k,l}e^{2l}\right)\varrho^k\right)\sin^{2n}(\psi)
\end{align}
where
\begin{equation}
d^{\cos}_{n',k',l'}=\sum_{i=1}^{\min(k',2n'+1)}\sum_{j=\max(0,\lfloor i/2\rfloor-l'+n')}^{\min(\lfloor i/2\rfloor,n'-\lceil i/2\rceil)}\binom{\lfloor i/2\rfloor}{j}\frac{(-1)^{\lceil i/2\rceil+j}}{i!} d^\phi_{n'-\lceil i/2\rceil-j,k',l',i}\label{eq:d_cos}
\end{equation}
and
\begin{equation}
\frac{\cos(\phi)}{\cos(\psi)}-1=\sum_{n=0}^\infty\left( \sum_{k=1}^\infty\left(\sum_{l=\max(n,k)}^\infty c^{\cos}_{n,k,l}e^{2l}\right)\varrho^k\right)\cos(2n\psi)
\end{equation}
where
\begin{equation}
c^{\cos}_{n,k,l}=\sum_{i=\max(n,l-k+1)}^l d^{\cos}_{i,k,l}\frac{2}{2^{2i+\delta_n}}(-1)^{n}\binom{2i}{i-n}\label{eq:c_cos}
\end{equation}

Note, the reason \eqref{eq:sin_split} and \eqref{eq:cos_split} are spelt out is that if both $\sin(\phi)$ and $\cos(\phi)$ are to be computed, it may be preferable to use them. The reason is that~\eqref{eq:sin_split} and \eqref{eq:cos_split} are essentially univariate polynomials whereas~\eqref{eq:sin} and~\eqref{eq:cos} are bivariate polynomials. In addition, apart from the factor $-\tfrac{\smash{\sin^{{}_2}(\psi)}}{\smash{\cos^{{}_2}(\psi)}}$, they are identical. See~\cite{Nilsson2024minimax} for an example of such usage and some more comments about it.

\subsection{Expansions for $\cos(\phi-\psi)$}
Trivially
\begin{equation}
\cos(x) = 1 + \sum_{i=1}^\infty\frac{(-1)^i}{(2i)!}x^{2i}
\end{equation}
Combining with \eqref{eq:phiDiffPow}
\begin{align}
\nonumber&\cos(\phi-\psi)\\
\nonumber&=1\!+\!\sum_{i=1}^\infty\frac{(-1)^i}{(2i)!}(1\!-\!\sin^2(\phi))^i\sin^{2i}(\phi)\!\sum_{n=0}^\infty \!\left(\sum_{k=2i}^\infty\!\left(\sum_{l=\max(n+2i,k)}^{n+k}\!\!\!\!\!\! d^\phi_{n,k,l,2i}e^{2l}\!\right)\!\varrho^k\!\right)\!\sin^{2n}(\psi)\\
\nonumber&=1\!+\!\sum_{i=1}^\infty\sum_{j=0}^i\frac{(-1)^{i+j}}{(2i)!}\binom{i}{j}\sum_{n=0}^\infty\left(\sum_{k=2i}^\infty\left(\sum_{l=\max(n+2i,k)}^{n+k} e^{2l}d^\phi_{n,k,l,2i}\right)\!\varrho^k\!\right)\sin^{2n+2i+2j}(\phi)\\
\raisetag{-9mm}\label{eq:cosDiff}&\quad\quad\quad=1 + \sum_{n=1}^\infty\left(\sum_{k=2}^\infty\left(\sum_{l=\max(n,k)}^{n-1+k} d'_{n,k,l}e^{2l}\right)\varrho^k\right)\sin^{2n}(\phi)
\end{align}
where
\begin{align*}
d'_{n',k',l'}&=\sum_{i=1}^\infty\sum_{j=0}^i\!\frac{(-1)^{i+j}}{(2i)!}\binom{i}{j}\!\sum_{n=0}^\infty\!\left(\sum_{k=2i}^\infty\left(\sum_{l=\max(n+2i,k)}^{n+k}\!\!\!\!\! \delta_{l-l'}d^\phi_{n,k,l,2i}\!\right)\!\delta_{k-k'}\!\right)\!\delta_{n+i+j-n'}\\
&=\sum_{i=1}^{\min(n',\lfloor k'/2 \rfloor)}\sum_{j=0}^{\min(i,n'-i)}\frac{(-1)^{i+j}}{(2i)!}\binom{i}{j}d^\phi_{n'-i-j,k',l',2i}
\end{align*}
The last equalities of the two proceeding equations are based on
\begin{itemize}
\item $\delta_{n+i+j-n'}$, $n\geq 0$ $i\geq 1$ and $0\leq j\leq i$ imply that
\begin{itemize}
\item $n'\geq 1$
\item $n = n'-i-j$
\item $0\leq j\leq n'-i$
\item $1\leq i\leq n'$
\end{itemize}
\item $\delta_{k-k'}$, $i\geq 1$, $k\geq 2i$ and $i\geq 1$ imply that
\begin{itemize}
\item $k'=k$
\item $k'\geq 2$
\item $i\leq \lfloor k'/2\rfloor$
\end{itemize}
\item $\delta_{l-l'}$, $\max(n+2i,k)\leq l\leq n+k$ and $j\leq i$ imply that
\begin{itemize}
\item $l'=l$
\item $\max(n',k') \leq l' \leq n'-1+k'$
\end{itemize}
\end{itemize}

\subsection{Expansions for $(1-e^2\sin^2(\phi))^{1/2}$}
Applying binomial expansion twice
\begin{align*}
(1-e^2\sin^2(\phi))^{1/2}&=1 + \sum_{i=1}^\infty\binom{1/2}{i}(-1)^ie^{2i}\sin^{2i}(\phi)\\
&=1 + \sum_{i=1}^\infty\binom{1/2}{i}(-1)^ie^{2i}\sin^{2i}(\psi)\left(1+\left(\frac{\sin(\phi)}{\sin(\psi)}-1\right)\right)^{2i}\\
&=1 + \sum_{i=1}^\infty\binom{1/2}{i}(-1)^ie^{2i}\sin^{2i}(\psi)\sum_{j=0}^{2i}\binom{2i}{j}\left(\frac{\sin(\phi)}{\sin(\psi)}-1\right)^{j}
\end{align*}
Similar to the series expansion for $(\phi-\psi)^i$~\eqref{eq:phiDiffPow}, the last factor can be expanded with $\max(n+1,k)$ replaced with $\max(n,k)$, i.e.
\begin{align*}
\left(\frac{\sin(\phi)}{\sin(\psi)}-1\right)^j&=\left(\sum_{n=0}^\infty \left(\sum_{k=1}^\infty\left(\sum_{l=\max(n,k)}^{n+k} d^{\sin}_{n,k,l}e^{2l}\right)\varrho^k\right)\sin^{2n}(\psi)\right)^{j}\\
&=\sum_{n=0}^\infty \left(\sum_{k=j}^\infty\left(\sum_{l=\max(n,k)}^{n+k} d^{\sin}_{n,k,l,j}e^{2l}\right)\varrho^k\right)\sin^{2n}(\psi)
\end{align*}
where
\begin{align*}
d_{n,k,l,i}^{\sin}&=[e^{2l}][\varrho^k]\left(\rule{0cm}{7mm}A_{n,i}(a_0,\dots,a_n)\left[\rule{0cm}{7mm}a_n^{i_n}\rightarrow\right.\right.\\
&\!\!\!\!\left.\left.\sum_{k=i_n}^\infty \hat{B}_{k,i_n}\left(\sum_{l=\max(1,n)}^{n+1} d^{\sin}_{n,1,l} e^{2l},\dots,\sum_{l=\max(n,k-i_n+1)}^{n+k-i_n+1} d^{\sin}_{n,k-i_n+1,l} e^{2l}\right)\varrho^k\right]\right)
\end{align*}
This in turn gives
\begin{align*}
&(1-e^2\sin^2(\phi))^{1/2}\\
&=\sum_{i=1}^\infty\!\binom{1/2}{i}(-1)^ie^{2i}\sin^{2i}(\psi)\sum_{j=0}^{2i}\!\binom{2i}{j}\sum_{n=0}^\infty \left(\sum_{k=j}^\infty\left(\sum_{l=\max(n,k)}^{n+k}\!\!\!\!\! d^{\sin}_{n,k,l,j}e^{2l}\!\right)\!\varrho^k\right)\!\sin^{2n}(\psi)\\
&=\sum_{i=1}^\infty\binom{1/2}{i}(-1)^i\sum_{j=0}^{2i}\binom{2i}{j}\sum_{n=0}^\infty \left(\sum_{k=j}^\infty\left(\sum_{l=\max(n,k)}^{n+k} d^{\sin}_{n,k,l,j}e^{2(l+i)}\right)\varrho^k\right)\sin^{2(n+i)}(\psi)
\end{align*}
which may be written as a plain triple sum
\begin{equation}
(1-e^2\sin^2(\phi))^{1/2}=\sum_{n=1}^\infty \left(\sum_{k=0}^\infty\left(\sum_{l=\max(n,k+1)}^{n+k} d^{N}_{n,k,l}e^{2l}\right)\varrho^k\right)\sin^{2n}(\psi)\label{eq:N}
\end{equation}
where
\begin{align*}
d^N_{n',k',l'}&=\!\sum_{i=1}^\infty\!\binom{1/2}{i}(-1)^i\sum_{j=0}^{2i}\!\binom{2i}{j}\!\sum_{n=0}^\infty \!\left(\sum_{k=j}^\infty\left(\sum_{l=\max(n,k)}^{n+k}\!\!\!\!\! \delta_{l+i-l'}d^{\sin}_{n,k,l,j}\!\right)\!\delta_{k-k'}\!\right)\!\delta_{n+i-n'}\\
&=\sum_{i=1}^{\min(n',l')}\binom{1/2}{i}(-1)^i\sum_{j=0}^{\min(2i,k')}\binom{2i}{j}d^{\sin}_{n'-i,k',l'-i,j}
\end{align*}
The last equalities of the two proceeding equations are based on
\begin{itemize}
\item $\delta_{n+i-n'}$, $n\geq 0$ and $i\geq1$  imply
\begin{itemize}
\item $n'\geq 1$
\item $1\leq i\leq n'$
\item $n=n'-i$
\end{itemize}
\item $\delta_{k-k'}$ and $k\geq j$ imply
\begin{itemize}
\item $k'=k$
\item $k'\geq 0$
\item $j\leq k'$
\end{itemize}
\item $\delta_{l+i-l'}$, $\max(n,k)\leq l\leq n+k$ and $j\leq 2i$ imply
\begin{itemize}
\item $l=l'-i$
\item $i= l'-l\leq l'$
\item $\max(n'-i,k')\leq l'-i$, i.e. $\max(n',k'+1)\leq l'$
\item $l'-i\leq n'-i+k'$, i.e. $l'\leq n'+k'$
\end{itemize}
\end{itemize}

\subsection{Expansions for $h$}\label{app:h}
Multiplying \eqref{eq:base1} with $\cos(\phi)$ and \eqref{eq:base2} with $\sin(\phi)$, adding, simplifying and solving for $h$ gives 
\begin{equation*}
h=\rho\cos(\phi-\psi)-a(1-e^2\sin^2(\phi))^{1/2}
\end{equation*}
Combining with \eqref{eq:cosDiff} and \eqref{eq:N} directly gives
\begin{equation}
\frac{h + a - \varrho}{a} = \sum_{n=1}^\infty\left(\sum_{k=0}^\infty\left(\sum_{l=\max(n,k+1)}^{n+k} d^h_{n,k,l}e^{2l}\right)\varrho^k\right)\sin^{2n}(\phi)
\end{equation}
where
\begin{equation}
d^h_{n,k,l}=
\begin{cases}
-d^N_{n,k,l} & k=0 \\
d'_{n,k+1,l}-d^N_{n,k,l} & \text{otherwise}
\end{cases}\label{eq:d_h}
\end{equation}
Using~\eqref{eq:sinPowToCosMul}, the Fourier series follows as
\begin{equation}
\frac{h + a - \varrho}{a}=\sum_{n=0}^\infty\left( \sum_{k=0}^\infty\left(\sum_{l=\max(n,k+1)}^\infty c^h_{n,k,l}e^{2l}\right)\varrho^k\right)\cos(2n\psi)
\end{equation}
where
\begin{equation}
c^h_{n,k,l}=\sum_{i=\max(n,l-k)}^l d^h_{i,k,l}\frac{2}{2^{2i+\delta_n}}(-1)^{n}\binom{2i}{i-n}\label{eq:c_h}
\end{equation}


\section{Tabulated coefficients}\label{sec::tabulated_coefficients}
The rational series expansion coefficients, i.e. \eqref{eq:c_phi}, \eqref{eq:d_phi}, \eqref{eq:c_h}, \eqref{eq:d_h}, \eqref{eq:c_sin}, \eqref{eq:d_sin}, \eqref{eq:c_cos}, and \eqref{eq:d_cos}, are tabulated in Table~\ref{tab:c_phi}-\ref{tab:d_cos}. 

\input{c_phi_table.tex}
\input{d_phi_table.tex}
\input{c_h_table.tex}
\input{d_h_table.tex}
\input{c_sin_table.tex}
\input{d_sin_table.tex}
\input{c_cos_table.tex}
\input{d_cos_table.tex}

\end{document}

%% file: c_phi_table.tex
\begin{table}
\tiny
\begin{tabular}{ @{\hspace{0mm}} c@{\hspace{1mm}} c@{\hspace{1mm}} | @{\hspace{0mm}} l@{}  l@{}  l@{}  l@{}  l@{}  l@{}  l@{}  l@{} }
$n$ & $k\setminus l$ & $\phantom{-}e^{2\cdot1}$ & $\phantom{-}e^{2\cdot2}$ & $\phantom{-}e^{2\cdot3}$ & $\phantom{-}e^{2\cdot4}$ & $\phantom{-}e^{2\cdot5}$ & $\phantom{-}e^{2\cdot6}$ & $\phantom{-}e^{2\cdot7}$ & $\phantom{-}e^{2\cdot8}$ \\ \hline
 1 & $\rho^{-1}$  & $\phantom{-}1/2$  & $\phantom{-}1/8$  & $\phantom{-}15/256$  & $\phantom{-}35/1024$  & $\phantom{-}735/32768$  & $\phantom{-}2079/131072$  & $\phantom{-}99099/8388608$  & $\phantom{-}306735/33554432$  \\[-0.2mm]
 & $\rho^{-2}$  & $\phantom{-}\times$  & $\phantom{-}0$  & $\phantom{-}1/32$  & $\phantom{-}1/32$  & $\phantom{-}7/256$  & $\phantom{-}3/128$  & $\phantom{-}165/8192$  & $\phantom{-}143/8192$  \\[-0.2mm]
 & $\rho^{-3}$  & $\phantom{-}\times$  & $\phantom{-}\times$  & $-3/64$  & $-9/256$  & $-105/4096$  & $-315/16384$  & $-31185/2097152$  & $-99099/8388608$  \\[-0.2mm]
 & $\rho^{-4}$  & $\phantom{-}\times$  & $\phantom{-}\times$  & $\phantom{-}\times$  & $\phantom{-}0$  & $\phantom{-}0$  & $\phantom{-}0$  & $\phantom{-}0$  & $\phantom{-}0$  \\[-0.2mm]
 & $\rho^{-5}$  & $\phantom{-}\times$  & $\phantom{-}\times$  & $\phantom{-}\times$  & $\phantom{-}\times$  & $-5/2048$  & $-25/8192$  & $-1575/524288$  & $-5775/2097152$  \\[-0.2mm]
 & $\rho^{-6}$  & $\phantom{-}\times$  & $\phantom{-}\times$  & $\phantom{-}\times$  & $\phantom{-}\times$  & $\phantom{-}\times$  & $\phantom{-}0$  & $\phantom{-}0$  & $\phantom{-}0$  \\[-0.2mm]
 & $\rho^{-7}$  & $\phantom{-}\times$  & $\phantom{-}\times$  & $\phantom{-}\times$  & $\phantom{-}\times$  & $\phantom{-}\times$  & $\phantom{-}\times$  & $-35/131072$  & $-245/524288$  \\
\end{tabular}
\caption{Coefficients $c^{\phi}_{n,k,l}$~\eqref{eq:c_phi} for $\phi-\psi$ in terms of $\sin(2n\psi)$. Note, the coefficients provided in~\cite{Morrison1961} are contained in the table.}
\label{tab:c_phi}
\end{table}

%% file: d_phi_table.tex
\begin{table}
\tiny
\begin{tabular}{ @{\hspace{0mm}} c@{\hspace{1mm}} c@{\hspace{1mm}} | @{\hspace{0mm}} l@{}  l@{}  l@{}  l@{}  l@{}  l@{}  l@{}  l@{}  l@{} }
$n$ & $k\setminus l$ & $\phantom{-}e^{2\cdot1}$ & $\phantom{-}e^{2\cdot2}$ & $\phantom{-}e^{2\cdot3}$ & $\phantom{-}e^{2\cdot4}$ & $\phantom{-}e^{2\cdot5}$ & $\phantom{-}e^{2\cdot6}$ & $\phantom{-}e^{2\cdot7}$ & $\phantom{-}e^{2\cdot8}$ & $\phantom{-}e^{2\cdot9}$ \\ \hline
 0 & $\rho^{-1}$  & $\phantom{-}1$  & $\phantom{-}\times$  & $\phantom{-}\times$  & $\phantom{-}\times$  & $\phantom{-}\times$  & $\phantom{-}\times$  & $\phantom{-}\times$  & $\phantom{-}\times$  & $\phantom{-}\times$  \\[-0.2mm]
 & $\rho^{-2}$  & $\phantom{-}\times$  & $\phantom{-}1$  & $\phantom{-}\times$  & $\phantom{-}\times$  & $\phantom{-}\times$  & $\phantom{-}\times$  & $\phantom{-}\times$  & $\phantom{-}\times$  & $\phantom{-}\times$  \\[-0.2mm]
 & $\rho^{-3}$  & $\phantom{-}\times$  & $\phantom{-}\times$  & $\phantom{-}1$  & $\phantom{-}\times$  & $\phantom{-}\times$  & $\phantom{-}\times$  & $\phantom{-}\times$  & $\phantom{-}\times$  & $\phantom{-}\times$  \\[-0.2mm]
 & $\rho^{-4}$  & $\phantom{-}\times$  & $\phantom{-}\times$  & $\phantom{-}\times$  & $\phantom{-}1$  & $\phantom{-}\times$  & $\phantom{-}\times$  & $\phantom{-}\times$  & $\phantom{-}\times$  & $\phantom{-}\times$  \\[-0.2mm]
 & $\rho^{-5}$  & $\phantom{-}\times$  & $\phantom{-}\times$  & $\phantom{-}\times$  & $\phantom{-}\times$  & $\phantom{-}1$  & $\phantom{-}\times$  & $\phantom{-}\times$  & $\phantom{-}\times$  & $\phantom{-}\times$  \\[-0.2mm]
 & $\rho^{-6}$  & $\phantom{-}\times$  & $\phantom{-}\times$  & $\phantom{-}\times$  & $\phantom{-}\times$  & $\phantom{-}\times$  & $\phantom{-}1$  & $\phantom{-}\times$  & $\phantom{-}\times$  & $\phantom{-}\times$  \\[-0.2mm]
 & $\rho^{-7}$  & $\phantom{-}\times$  & $\phantom{-}\times$  & $\phantom{-}\times$  & $\phantom{-}\times$  & $\phantom{-}\times$  & $\phantom{-}\times$  & $\phantom{-}1$  & $\phantom{-}\times$  & $\phantom{-}\times$  \\
\end{tabular}
\caption{Coefficients $d^{\phi}_{n,k,l}$~\eqref{eq:d_phi} for $(\phi-\psi)/(\cos(\psi)\sin(\psi))$ in terms of $\sin^{2n}(\psi)$.}
\label{tab:d_phi}
\end{table}

%% file: c_h_table.tex
\begin{table}
\tiny
\begin{tabular}{ @{\hspace{0mm}} c@{\hspace{1mm}} c@{\hspace{1mm}} | @{\hspace{0mm}} l@{}  l@{}  l@{}  l@{}  l@{}  l@{}  l@{}  l@{} }
$n$ & $k\setminus l$ & $\phantom{-}e^{2\cdot1}$ & $\phantom{-}e^{2\cdot2}$ & $\phantom{-}e^{2\cdot3}$ & $\phantom{-}e^{2\cdot4}$ & $\phantom{-}e^{2\cdot5}$ & $\phantom{-}e^{2\cdot6}$ & $\phantom{-}e^{2\cdot7}$ & $\phantom{-}e^{2\cdot8}$ \\ \hline
 0 & $\rho^{-0}$  & $\phantom{-}1/4$  & $\phantom{-}3/64$  & $\phantom{-}5/256$  & $\phantom{-}175/16384$  & $\phantom{-}441/65536$  & $\phantom{-}4851/1048576$  & $\phantom{-}14157/4194304$  & $\phantom{-}2760615/1073741824$  \\[-0.2mm]
 & $\rho^{-1}$  & $\phantom{-}\times$  & $\phantom{-}1/16$  & $\phantom{-}1/32$  & $\phantom{-}5/256$  & $\phantom{-}7/512$  & $\phantom{-}21/2048$  & $\phantom{-}33/4096$  & $\phantom{-}429/65536$  \\[-0.2mm]
 & $\rho^{-2}$  & $\phantom{-}\times$  & $\phantom{-}\times$  & $\phantom{-}0$  & $\phantom{-}0$  & $\phantom{-}0$  & $\phantom{-}0$  & $\phantom{-}0$  & $\phantom{-}0$  \\[-0.2mm]
 & $\rho^{-3}$  & $\phantom{-}\times$  & $\phantom{-}\times$  & $\phantom{-}\times$  & $\phantom{-}1/1024$  & $\phantom{-}1/1024$  & $\phantom{-}7/8192$  & $\phantom{-}3/4096$  & $\phantom{-}165/262144$  \\[-0.2mm]
 & $\rho^{-4}$  & $\phantom{-}\times$  & $\phantom{-}\times$  & $\phantom{-}\times$  & $\phantom{-}\times$  & $\phantom{-}0$  & $\phantom{-}0$  & $\phantom{-}0$  & $\phantom{-}0$  \\[-0.2mm]
 & $\rho^{-5}$  & $\phantom{-}\times$  & $\phantom{-}\times$  & $\phantom{-}\times$  & $\phantom{-}\times$  & $\phantom{-}\times$  & $\phantom{-}1/16384$  & $\phantom{-}3/32768$  & $\phantom{-}27/262144$  \\[-0.2mm]
 & $\rho^{-6}$  & $\phantom{-}\times$  & $\phantom{-}\times$  & $\phantom{-}\times$  & $\phantom{-}\times$  & $\phantom{-}\times$  & $\phantom{-}\times$  & $\phantom{-}0$  & $\phantom{-}0$  \\
\end{tabular}
\caption{Coefficients $c^{h}_{n,k,l}$~\eqref{eq:c_h} for $(h+a-\rho)/a$ in terms of $\cos(2n\psi)$.}
\label{tab:c_h}
\end{table}

%% file: d_h_table.tex
\begin{table}
\tiny
\begin{tabular}{ @{\hspace{0mm}} c@{\hspace{1mm}} c@{\hspace{1mm}} | @{\hspace{0mm}} l@{}  l@{}  l@{}  l@{}  l@{}  l@{}  l@{}  l@{} }
$n$ & $k\setminus l$ & $\phantom{-}e^{2\cdot1}$ & $\phantom{-}e^{2\cdot2}$ & $\phantom{-}e^{2\cdot3}$ & $\phantom{-}e^{2\cdot4}$ & $\phantom{-}e^{2\cdot5}$ & $\phantom{-}e^{2\cdot6}$ & $\phantom{-}e^{2\cdot7}$ & $\phantom{-}e^{2\cdot8}$ \\ \hline
 1 & $\rho^{-0}$  & $\phantom{-}1/2$  & $\phantom{-}\times$  & $\phantom{-}\times$  & $\phantom{-}\times$  & $\phantom{-}\times$  & $\phantom{-}\times$  & $\phantom{-}\times$  & $\phantom{-}\times$  \\[-0.2mm]
 & $\rho^{-1}$  & $\phantom{-}\times$  & $\phantom{-}1/2$  & $\phantom{-}\times$  & $\phantom{-}\times$  & $\phantom{-}\times$  & $\phantom{-}\times$  & $\phantom{-}\times$  & $\phantom{-}\times$  \\[-0.2mm]
 & $\rho^{-2}$  & $\phantom{-}\times$  & $\phantom{-}\times$  & $\phantom{-}1/2$  & $\phantom{-}\times$  & $\phantom{-}\times$  & $\phantom{-}\times$  & $\phantom{-}\times$  & $\phantom{-}\times$  \\[-0.2mm]
 & $\rho^{-3}$  & $\phantom{-}\times$  & $\phantom{-}\times$  & $\phantom{-}\times$  & $\phantom{-}1/2$  & $\phantom{-}\times$  & $\phantom{-}\times$  & $\phantom{-}\times$  & $\phantom{-}\times$  \\[-0.2mm]
 & $\rho^{-4}$  & $\phantom{-}\times$  & $\phantom{-}\times$  & $\phantom{-}\times$  & $\phantom{-}\times$  & $\phantom{-}1/2$  & $\phantom{-}\times$  & $\phantom{-}\times$  & $\phantom{-}\times$  \\[-0.2mm]
 & $\rho^{-5}$  & $\phantom{-}\times$  & $\phantom{-}\times$  & $\phantom{-}\times$  & $\phantom{-}\times$  & $\phantom{-}\times$  & $\phantom{-}1/2$  & $\phantom{-}\times$  & $\phantom{-}\times$  \\[-0.2mm]
 & $\rho^{-6}$  & $\phantom{-}\times$  & $\phantom{-}\times$  & $\phantom{-}\times$  & $\phantom{-}\times$  & $\phantom{-}\times$  & $\phantom{-}\times$  & $\phantom{-}1/2$  & $\phantom{-}\times$  \\
\end{tabular}
\caption{Coefficients $d^{h}_{n,k,l}$~\eqref{eq:d_h} for $(h+a-\rho)/a$ in terms of $\sin^{2n}(\psi)$.}
\label{tab:d_h}
\end{table}

%% file: c_sin_table.tex
\begin{table}
\tiny
\begin{tabular}{ @{\hspace{0mm}} c@{\hspace{1mm}} c@{\hspace{1mm}} | @{\hspace{0mm}} l@{}  l@{}  l@{}  l@{}  l@{}  l@{}  l@{}  l@{} }
$n$ & $k\setminus l$ & $\phantom{-}e^{2\cdot1}$ & $\phantom{-}e^{2\cdot2}$ & $\phantom{-}e^{2\cdot3}$ & $\phantom{-}e^{2\cdot4}$ & $\phantom{-}e^{2\cdot5}$ & $\phantom{-}e^{2\cdot6}$ & $\phantom{-}e^{2\cdot7}$ & $\phantom{-}e^{2\cdot8}$ \\ \hline
 0 & $\rho^{-1}$  & $\phantom{-}1/2$  & $\phantom{-}1/16$  & $\phantom{-}3/128$  & $\phantom{-}25/2048$  & $\phantom{-}245/32768$  & $\phantom{-}1323/262144$  & $\phantom{-}7623/2097152$  & $\phantom{-}184041/67108864$  \\[-0.2mm]
 & $\rho^{-2}$  & $\phantom{-}\times$  & $\phantom{-}3/16$  & $\phantom{-}1/32$  & $\phantom{-}3/256$  & $\phantom{-}3/512$  & $\phantom{-}7/2048$  & $\phantom{-}9/4096$  & $\phantom{-}99/65536$  \\[-0.2mm]
 & $\rho^{-3}$  & $\phantom{-}\times$  & $\phantom{-}\times$  & $\phantom{-}1/8$  & $\phantom{-}1/32$  & $\phantom{-}15/1024$  & $\phantom{-}35/4096$  & $\phantom{-}735/131072$  & $\phantom{-}2079/524288$  \\[-0.2mm]
 & $\rho^{-4}$  & $\phantom{-}\times$  & $\phantom{-}\times$  & $\phantom{-}\times$  & $\phantom{-}85/1024$  & $\phantom{-}23/1024$  & $\phantom{-}87/8192$  & $\phantom{-}25/4096$  & $\phantom{-}1025/262144$  \\[-0.2mm]
 & $\rho^{-5}$  & $\phantom{-}\times$  & $\phantom{-}\times$  & $\phantom{-}\times$  & $\phantom{-}\times$  & $\phantom{-}1/16$  & $\phantom{-}5/256$  & $\phantom{-}21/2048$  & $\phantom{-}105/16384$  \\[-0.2mm]
 & $\rho^{-6}$  & $\phantom{-}\times$  & $\phantom{-}\times$  & $\phantom{-}\times$  & $\phantom{-}\times$  & $\phantom{-}\times$  & $\phantom{-}791/16384$  & $\phantom{-}519/32768$  & $\phantom{-}2217/262144$  \\[-0.2mm]
 & $\rho^{-7}$  & $\phantom{-}\times$  & $\phantom{-}\times$  & $\phantom{-}\times$  & $\phantom{-}\times$  & $\phantom{-}\times$  & $\phantom{-}\times$  & $\phantom{-}5/128$  & $\phantom{-}7/512$  \\
\end{tabular}
\caption{Coefficients $c^{\sin}_{n,k,l}$~\eqref{eq:c_sin} for $\sin(\phi)/\sin(\psi)-1$ in terms of $\cos(2n\psi)$.}
\label{tab:c_sin}
\end{table}

%% file: d_sin_table.tex
\begin{table}
\tiny
\begin{tabular}{ @{\hspace{0mm}} c@{\hspace{1mm}} c@{\hspace{1mm}} | @{\hspace{0mm}} l@{}  l@{}  l@{}  l@{}  l@{}  l@{}  l@{}  l@{} }
$n$ & $k\setminus l$ & $\phantom{-}e^{2\cdot1}$ & $\phantom{-}e^{2\cdot2}$ & $\phantom{-}e^{2\cdot3}$ & $\phantom{-}e^{2\cdot4}$ & $\phantom{-}e^{2\cdot5}$ & $\phantom{-}e^{2\cdot6}$ & $\phantom{-}e^{2\cdot7}$ & $\phantom{-}e^{2\cdot8}$ \\ \hline
 0 & $\rho^{-1}$  & $\phantom{-}1$  & $\phantom{-}\times$  & $\phantom{-}\times$  & $\phantom{-}\times$  & $\phantom{-}\times$  & $\phantom{-}\times$  & $\phantom{-}\times$  & $\phantom{-}\times$  \\[-0.2mm]
 & $\rho^{-2}$  & $\phantom{-}\times$  & $\phantom{-}1$  & $\phantom{-}\times$  & $\phantom{-}\times$  & $\phantom{-}\times$  & $\phantom{-}\times$  & $\phantom{-}\times$  & $\phantom{-}\times$  \\[-0.2mm]
 & $\rho^{-3}$  & $\phantom{-}\times$  & $\phantom{-}\times$  & $\phantom{-}1$  & $\phantom{-}\times$  & $\phantom{-}\times$  & $\phantom{-}\times$  & $\phantom{-}\times$  & $\phantom{-}\times$  \\[-0.2mm]
 & $\rho^{-4}$  & $\phantom{-}\times$  & $\phantom{-}\times$  & $\phantom{-}\times$  & $\phantom{-}1$  & $\phantom{-}\times$  & $\phantom{-}\times$  & $\phantom{-}\times$  & $\phantom{-}\times$  \\[-0.2mm]
 & $\rho^{-5}$  & $\phantom{-}\times$  & $\phantom{-}\times$  & $\phantom{-}\times$  & $\phantom{-}\times$  & $\phantom{-}1$  & $\phantom{-}\times$  & $\phantom{-}\times$  & $\phantom{-}\times$  \\[-0.2mm]
 & $\rho^{-6}$  & $\phantom{-}\times$  & $\phantom{-}\times$  & $\phantom{-}\times$  & $\phantom{-}\times$  & $\phantom{-}\times$  & $\phantom{-}1$  & $\phantom{-}\times$  & $\phantom{-}\times$  \\[-0.2mm]
 & $\rho^{-7}$  & $\phantom{-}\times$  & $\phantom{-}\times$  & $\phantom{-}\times$  & $\phantom{-}\times$  & $\phantom{-}\times$  & $\phantom{-}\times$  & $\phantom{-}1$  & $\phantom{-}\times$  \\
\end{tabular}
\caption{Coefficients $d^{\sin}_{n,k,l}$~\eqref{eq:d_sin} for $\sin(\phi)/\sin(\psi)-1$ in terms of $\sin^{2n}(\psi)$.}
\label{tab:d_sin}
\end{table}

%% file: c_cos_table.tex
\begin{table}
\tiny
\begin{tabular}{ @{\hspace{0mm}} c@{\hspace{1mm}} c@{\hspace{1mm}} | @{\hspace{0mm}} l@{}  l@{}  l@{}  l@{}  l@{}  l@{}  l@{}  l@{} }
$n$ & $k\setminus l$ & $\phantom{-}e^{2\cdot1}$ & $\phantom{-}e^{2\cdot2}$ & $\phantom{-}e^{2\cdot3}$ & $\phantom{-}e^{2\cdot4}$ & $\phantom{-}e^{2\cdot5}$ & $\phantom{-}e^{2\cdot6}$ & $\phantom{-}e^{2\cdot7}$ & $\phantom{-}e^{2\cdot8}$ \\ \hline
 0 & $\rho^{-1}$  & $-1/2$  & $-3/16$  & $-15/128$  & $-175/2048$  & $-2205/32768$  & $-14553/262144$  & $-99099/2097152$  & $-2760615/67108864$  \\[-0.2mm]
 & $\rho^{-2}$  & $\phantom{-}\times$  & $\phantom{-}3/16$  & $\phantom{-}5/32$  & $\phantom{-}35/256$  & $\phantom{-}63/512$  & $\phantom{-}231/2048$  & $\phantom{-}429/4096$  & $\phantom{-}6435/65536$  \\[-0.2mm]
 & $\rho^{-3}$  & $\phantom{-}\times$  & $\phantom{-}\times$  & $-1/8$  & $-5/32$  & $-175/1024$  & $-735/4096$  & $-24255/131072$  & $-99099/524288$  \\[-0.2mm]
 & $\rho^{-4}$  & $\phantom{-}\times$  & $\phantom{-}\times$  & $\phantom{-}\times$  & $\phantom{-}85/1024$  & $\phantom{-}147/1024$  & $\phantom{-}1575/8192$  & $\phantom{-}957/4096$  & $\phantom{-}70785/262144$  \\[-0.2mm]
 & $\rho^{-5}$  & $\phantom{-}\times$  & $\phantom{-}\times$  & $\phantom{-}\times$  & $\phantom{-}\times$  & $-1/16$  & $-35/256$  & $-441/2048$  & $-4851/16384$  \\[-0.2mm]
 & $\rho^{-6}$  & $\phantom{-}\times$  & $\phantom{-}\times$  & $\phantom{-}\times$  & $\phantom{-}\times$  & $\phantom{-}\times$  & $\phantom{-}791/16384$  & $\phantom{-}4227/32768$  & $\phantom{-}61545/262144$  \\[-0.2mm]
 & $\rho^{-7}$  & $\phantom{-}\times$  & $\phantom{-}\times$  & $\phantom{-}\times$  & $\phantom{-}\times$  & $\phantom{-}\times$  & $\phantom{-}\times$  & $-5/128$  & $-63/512$  \\
\end{tabular}
\caption{Coefficients $c^{\cos}_{n,k,l}$~\eqref{eq:c_cos} for $\cos(\phi)/\cos(\psi)-1$ in terms of $\cos(2n\psi)$.}
\label{tab:c_cos}
\end{table}

%% file: d_cos_table.tex
\begin{table}
\tiny
\begin{tabular}{ @{\hspace{0mm}} c@{\hspace{1mm}} c@{\hspace{1mm}} | @{\hspace{0mm}} l@{}  l@{}  l@{}  l@{}  l@{}  l@{}  l@{}  l@{} }
$n$ & $k\setminus l$ & $\phantom{-}e^{2\cdot1}$ & $\phantom{-}e^{2\cdot2}$ & $\phantom{-}e^{2\cdot3}$ & $\phantom{-}e^{2\cdot4}$ & $\phantom{-}e^{2\cdot5}$ & $\phantom{-}e^{2\cdot6}$ & $\phantom{-}e^{2\cdot7}$ & $\phantom{-}e^{2\cdot8}$ \\ \hline
 1 & $\rho^{-1}$  & $-1$  & $\phantom{-}\times$  & $\phantom{-}\times$  & $\phantom{-}\times$  & $\phantom{-}\times$  & $\phantom{-}\times$  & $\phantom{-}\times$  & $\phantom{-}\times$  \\[-0.2mm]
 & $\rho^{-2}$  & $\phantom{-}\times$  & $-3/2$  & $\phantom{-}\times$  & $\phantom{-}\times$  & $\phantom{-}\times$  & $\phantom{-}\times$  & $\phantom{-}\times$  & $\phantom{-}\times$  \\[-0.2mm]
 & $\rho^{-3}$  & $\phantom{-}\times$  & $\phantom{-}\times$  & $-2$  & $\phantom{-}\times$  & $\phantom{-}\times$  & $\phantom{-}\times$  & $\phantom{-}\times$  & $\phantom{-}\times$  \\[-0.2mm]
 & $\rho^{-4}$  & $\phantom{-}\times$  & $\phantom{-}\times$  & $\phantom{-}\times$  & $-5/2$  & $\phantom{-}\times$  & $\phantom{-}\times$  & $\phantom{-}\times$  & $\phantom{-}\times$  \\[-0.2mm]
 & $\rho^{-5}$  & $\phantom{-}\times$  & $\phantom{-}\times$  & $\phantom{-}\times$  & $\phantom{-}\times$  & $-3$  & $\phantom{-}\times$  & $\phantom{-}\times$  & $\phantom{-}\times$  \\[-0.2mm]
 & $\rho^{-6}$  & $\phantom{-}\times$  & $\phantom{-}\times$  & $\phantom{-}\times$  & $\phantom{-}\times$  & $\phantom{-}\times$  & $-7/2$  & $\phantom{-}\times$  & $\phantom{-}\times$  \\[-0.2mm]
 & $\rho^{-7}$  & $\phantom{-}\times$  & $\phantom{-}\times$  & $\phantom{-}\times$  & $\phantom{-}\times$  & $\phantom{-}\times$  & $\phantom{-}\times$  & $-4$  & $\phantom{-}\times$  \\
\end{tabular}
\caption{Coefficients $d^{\cos}_{n,k,l}$~\eqref{eq:d_cos} for $\cos(\phi)/\cos(\psi)-1$ in terms of $\sin^{2n}(\psi)$.}
\label{tab:d_cos}
\end{table}